\documentclass[11pt]{article}

\usepackage{amsmath,geometry}
\usepackage{amsfonts,amssymb}
\usepackage{amsthm,authblk}
\usepackage{graphicx,cite,enumerate}
\usepackage{color}

\newtheorem{thm}{Theorem}[section]
\newtheorem{lem}[thm]{Lemma}
\newtheorem{prop}[thm]{Proposition}

\newtheorem{rem}[thm]{Remark}

\def\XXint#1#2#3{{\setbox0=\hbox{$#1{#2#3}{\int}$ }
\vcenter{\hbox{$#2#3$ }}\kern-.6\wd0}}

\newcommand{\RR}{\mathbb{R}}      
\newcommand{\N}{\mathbb{N}}         
\newcommand{\eps}{\varepsilon}        
\newcommand{\sprime}{^\prime}        
\newcommand{\dprime}{^{\prime\prime}}  
\newcommand{\rhom}{\rho_\text{min}}      
\newcommand{\rhoc}{\rho_{c1}}             
\newcommand{\sint}[3]{\int_{#1}{#2} \: d{#3}}  
\newcommand{\norm}{|x-y|^\alpha} 
\newcommand{\dv}{\Delta V}  

\begin{document}

\title{\bf On an isoperimetric problem with a competing non-local
  term: Quantitative results}

\author[1]{Cyrill B. Muratov} 
\author[1,2]{Anthony Zaleski}
\affil[1]{Department of Mathematical Sciences, New Jersey Institute of
  Technology, Newark, NJ 07102} 
\affil[2]{Department of Mathematics, Rutgers University, Piscataway,
  NJ 08854} 
\renewcommand\Authands{ and }
\renewcommand\Affilfont{\itshape\small}

\maketitle

\begin{abstract}
  \noindent This paper provides a quantitative version of the recent
  result of Kn\"upfer and Muratov ({\it Commun. Pure Appl. Math.} {\bf
    66} (2013), 1129--1162) concerning the solutions of an extension
  of the classical isoperimetric problem in which a non-local
  repulsive term involving Riesz potential is present. There it was
  shown that in two space dimensions the minimizer of the considered
  problem is either a ball or does not exist, depending on whether or
  not the volume constraint lies in an explicit interval around zero,
  provided that the Riesz kernel decays sufficiently slowly. Here we
  give an explicit estimate for the exponents of the Riesz kernel for
  which the result holds.
\end{abstract}


\section{Introduction} \label{sec:intro}

This paper is concerned with a study of the following non-local
extension of the classical isoperimetric problem: minimize the energy 
\begin{eqnarray}
  \label{E}
  E(F) = P(F) + V(F),
\end{eqnarray}
among all Lebesgue measurable sets $F \subset \mathbb R^n$, $n \geq
2$, subject to the constraint $|F| = m$, with some $m > 0$ fixed. Here
$P(F)$ is the perimeter of the set $F$ in the sense of De Giorgi
\cite{ambrosio}:
\begin{eqnarray}
  \label{P}
  P(F) := \sup \left\{ \int_F \nabla \cdot \phi \, dx : \ \phi \in
    C^1_c(\mathbb R^n; \mathbb R^n), \ |\phi| \leq 1 \right\},
\end{eqnarray}
and $V(F)$ describes a non-local repulsive interaction:
\begin{eqnarray}
  \label{V}
  V(F) := \int_F \int_F {1 \over |x - y|^\alpha} \, dx \, dy,
\end{eqnarray}
for some $\alpha \in (0, n)$. Minimizers of this problem will be
referred to as ``minimizers of $E$ with mass $m$'' in the rest of the
paper.

The variational problem above arises in a number of contexts of
mathematical physics, most notably in the case of $\alpha = 1$ and $n
= 3$, when it corresponds to the classical Gamow's liquid drop model
of an atomic nucleus \cite{gamow30,bohr36,bohr39}. Note that the case
$\alpha = 1$ is of particular importance, since it corresponds to the
repulsive Coulombic interaction and is, therefore, also relevant to a
wide variety of other physical situations, both in three and two space
dimensions (see the discussion in \cite{km:cpam13} for further
details). In particular, during their studies of a closely related
Ohta-Kawasaki energy functional of diblock copolymer systems
\cite{ohta86}, Choksi and Peletier asked if the minimizer of the above
problem with $n = 3$ and $\alpha = 1$ is a ball whenever it exists
\cite{choksi11} (see also \cite{choksi12} for an overview of the
problem and recent results). The answer to this question is not
obvious at all, since the two terms in the energy in \eqref{E} are in
direct competition with each other: while the first term tends to put
the mass together, the second term favors spreading the mass apart as
much as possible. As a consequence, a single ball is no longer a
minimizer if the value of $m$ becomes sufficiently large, since for
large enough values of $m$ it is advantageous to split the ball into
two balls of equal volume and move the resulting smaller balls far
apart. A far more difficult question, however, is whether the ball has
the best {\em shape} among all competitors for the energy minimizers
at a given $m > 0$.

A rather detailed study of the variational problem described above was
recently performed by Kn\"upfer and Muratov \cite{km:cpam12,km:cpam13}
(see also
\cite{lu14,cicalese13,julin12,bonacini13,acerbi13,sternberg11,topaloglu13,figalli14}
for some related recent work). Some basic existence and non-existence
properties of the minimizers of the considered variational problem
were established, together with the more detailed information about
the shape of the minimizers in certain parameter regimes
\cite{km:cpam12,km:cpam13}. We summarize those findings under a
simplifying assumption of $n \leq 3$, corresponding to the spatial
dimensionality of the problems of physical interest. In this case it
is known that a minimizer of $E$ with mass $m$ exists for all $m \leq
m_1$, for some $m_1 > 0$ depending on $\alpha$ and $n$. If,
furthermore, $\alpha < 2$, then there exist $m_0 > 0$ and $m_2 > 0$
depending on $\alpha$ and $n$ such that the minimizer is a ball
whenever $m \leq m_0$, and minimizers do not exist whenever $m > m_2$.

Clearly, if $\alpha < 2$ and $n \leq 3$, we have $0 < m_0 \leq m_1
\leq m_2 < \infty$. However, it is not obvious whether one could
choose $m_0 = m_1 = m_2$, indicating that minimizers are balls
whenever they exist. A gap between the values of $m_0$ and $m_1$ would
indicate existence of non-radial minimizers for certain values of
$m$. Similarly, a gap between the values of $m_1$ and $m_2$ would
indicate that the set of values of $m$ for which minimizers exist is
not a bounded interval around the origin. Both of these possibilities
would yield a negative answer to the question of Choksi and Peletier
for the problem under consideration. Moreover, even if $m_0 = m_1 =
m_2$, it is not yet clear that those values are equal to $m_{c1} > 0$,
the value of $m$ at which one ball of mass $m$ has the same energy as
two balls of mass $\tfrac12 m$ infinitely far apart, which is what one
would expect if the splitting mechanism into two equal size balls were
the dominant mechanism for reducing the energy at large masses.

Despite a general lack of understanding of the global structure of the
energy minimizers in the considered problem, some partial results
currently exist in the case of sufficiently slowly decaying kernels in
\eqref{V} \cite{km:cpam12,bonacini13}. In \cite{km:cpam12}, Kn\"upfer
and Muratov showed that when $n = 2$ there exists a universal constant
$\alpha_0 > 0$ such that the minimizer of the considered problem is a
ball whenever it exists, and one can choose $m_0 = m_1 = m_2$ for all
$\alpha < \alpha_0$. This result was recently extended by Bonacini and
Cristoferi to higher dimensions \cite{bonacini13}. At the same time,
it was also shown in \cite{km:cpam12} that when $n = 2$ and $\alpha
\leq \alpha_0$, one can, in fact, choose $m_0 = m_1 = m_2 = m_{c1}$,
where the value of $m_{c1}$ is explicitly given by
\begin{align}
  \label{mc1}
  m_{c1}(\alpha) = \pi \left( \frac{\left(\sqrt{2}-1\right) \Gamma
      \left(2-\frac{\alpha}{2}\right) \Gamma
      \left(3-\frac{\alpha}{2}\right)}{\pi \left(1 - 2^{\alpha - 2
          \over 2}\right) \Gamma (2-\alpha)} \right)^{2 \over 3 -
    \alpha},
\end{align}
where $\Gamma(z)$ is the Euler Gamma-function \cite{abramowitz},
confirming the global bifurcation picture suggested by Choksi and
Peletier \cite{choksi11,choksi12} in the case of $n = 2$ and $\alpha$
sufficiently small.

Whether such a picture remains valid for all $n \geq 2$ and all
$\alpha \in (0, n)$ is still far from clear. In particular, as a
starting point it would be interesting to know if one could choose
$\alpha_0 = 2$ for $n = 2$ in \cite[Theorem 2.7]{km:cpam12}. At the
very minimum, such a result would require a quantitative version of
the analysis of \cite{km:cpam12}. The goal of this paper is to provide
such an analysis. Here is our main result, which gives the following
quantitative version of this theorem.

\begin{thm}
  \label{t:main}
  Let $n = 2$, let $\alpha \leq 0.034$, and let $m_{c1}$ be given by
  \eqref{mc1}. Then minimizers of $E$ with mass $m$ exist if and only
  if $m \leq m_{c1}$, and every minimizer of $E$ is a disk of radius
  $\sqrt{m/\pi}$.
\end{thm}

The proof of Theorem \ref{t:main} mostly follows along the lines of
\cite{km:cpam12}, while keeping track of the constants appearing in
all the estimates. We recall that the strategy in proving
\cite[Theorem 2.7]{km:cpam12} was to demonstrate, for all $m \leq M$
with $M > 0$ fixed, that for small enough $\alpha > 0$ depending only
on $M$ the minimizer, if it exists, must be a convex set which is only
a small perturbation of the disk of radius $\sqrt{m/\pi}$ in the
Hausdorff sense. This is achieved by combining the quantitative
version of the isoperimetric inequality with suitable a priori upper
bounds for the energy, together with a careful analysis of the
rigidity of disks with respect to small perturbations. This result is
then combined with a non-existence result for minimizers with $m > M$
for some $M > 0$, which is uniform in $\alpha \ll 1$. Inevitably, this
strategy is guaranteed to work only for sufficiently small values of
$\alpha$. Yet, it is rather surprising that we were only able to prove
our result for such a narrow range of values of $\alpha$, despite our
attempts to strive for the best constants in the analysis wherever
possible. Perhaps this is an indication that the global bifurcation
structure of the considered variational problem may be more complex,
and further non-perturbative studies of the problem are needed.

The rest of the paper is organized as follows. In
Sec. \ref{sec:preliminaries}, we summarize several facts about
minimizers of $E$ with mass $m$ and state a few technical facts. In
Sec. \ref{sec:Ebound}, we derive a tight upper bound on the minimal
energy that scales linearly with $m$ for large masses. In
Sec. \ref{sec:nonexistence}, we give a quantitative version of the
non-existence result for large masses. In Sec. \ref{sec:shape}, we
give a quantitative criterion about when minimizers are balls whenever
they exist. Finally, in Sec. \ref{sec:results} we put all the obtained
estimates together and prove the main Theorem. In this section, we
also give a numerical estimate of the value of $\alpha_0 \approx
0.04273$ that is slightly better than our analytical estimate.

Throughout the paper, all the constants may depend implicitly on
$\alpha$ and $\eps$, a parameter related to $m$ that appears in
Sec. \ref{sec:shape}. These dependences will be suppressed whenever it
does not cause ambiguity in order to simplify the notation. The
algebraic computations were performed with the help of {\sc
  Mathematica 8.0} software.

\section{Preliminaries}
\label{sec:preliminaries}

We start by collecting some basic facts about the considered
variational problem. Even if we stated the original problem in general
spatial dimensionality, in two space dimensions the minimization
problem is equivalent to minimizing $E$ among open sets with a $C^1$
boundary. This is because of the basic regularity property that
minimizers of $E$ inherit from being quasi-minimizers of the perimeter
(see, e.g., \cite{ambrosio98,tamanini84}). We have the following basic
regularity result for the minimizers of $E$ (in the rest of the paper,
we always assume that $n = 2$).

\begin{prop}[\cite{km:cpam12}, Proposition 2.1 ] \label{prp-basic} %
  Let $m > 0$ and let $\Omega$ be a minimizer of $E$ among all open
  sets with boundary of class $C^1$ and $|\Omega| = m$. Then
  \begin{enumerate}[(i)]
  \item $\partial \Omega$ is of class $C^{2,\beta}$, for some
    $\beta \in (0,1)$, with the regularity constants depending only on
    $m$ and $\alpha$.

  \item $\Omega$ is bounded, connected and contains at most finitely
    many holes.

  \item $\partial \Omega$ satisfies the Euler-Lagrange equation
      \begin{align} \label{EL}
        \kappa(x) + 2 v(x) - \mu = 0, \qquad v(x) := \int_\Omega {1 \over |x
          - y|^\alpha} \ dy,
      \end{align}
      where $\kappa(x)$ is the curvature (positive if $\Omega$ is
      convex) at $x \in \partial \Omega$ and $\mu \in {\mathbb R}$.
  \end{enumerate}
\end{prop}

\noindent Note that since $\Omega$ in Proposition \ref{prp-basic} is
connected, we also have the following elementary bound:
\begin{align}
  \label{diamP}
  \text{diam}(\Omega) \leq \frac12 P(\Omega),
\end{align}
which will be repeatedly used throughout our paper.

Concerning the minimizers of $E$ (in a wider class of sets of finite
perimeter, also for any dimension $n \geq 2$ and any fixed $\alpha \in
(0, n)$), we know that their existence is guaranteed for all
sufficiently small values of $m$ \cite[Theorem 3.1]{km:cpam13}. For $n
= 2$, existence of minimizers in the sense of Proposition
\ref{prp-basic} then follows, possibly after a redefinition of
$\Omega$ on a set of Lebesgue measure zero \cite[Proposition
2.1]{km:cpam13}. In the context of the present paper, however, we have
the following quantitative improvement of the existence result in
\cite{km:cpam12}.

\begin{prop}
  \label{p:exist}
  Suppose that the minimizer of $E$ with mass $m$, whenever it exists,
  is a disk of radius $\sqrt{m/\pi}$ for all $m \leq m_0$ with some
  $m_0 > m_{c1}$, where $m_{c1}$ is defined in \eqref{mc1}. Then
  \begin{enumerate}[(i)]
  \item The minimizer exists and is a disk for all $m \leq m_{c1}$.
  \item There is no minimizer for all $m_{c1} < m \leq m_0$.
  \end{enumerate}
\end{prop}

\begin{proof}
  The proof follows from \cite[Lemma 3.6 and Proposition
  8.8]{km:cpam12}. Indeed, suppose that $m \leq m_0$. If the minimizer
  of $E$ with mass $m$ exists, it is a disk by assumption of the
  proposition. However, by \cite[Lemma 3.6]{km:cpam12} this is not
  possible if $m > m_{c1}$, yielding the second statement. To prove
  the first statement, suppose, by contradiction, that there is no
  minimizer and $m \leq m_{c1}$. Then by \cite[Proposition
  8.8]{km:cpam12} and the assumption of the proposition, for any
  measurable set $F \subset \mathbb R^2$ with $|F| = m$ there is a set
  $\tilde F = \cup_{i=1}^N B_{R_i}(x_i)$, a union of finitely many
  disjoint disks, such that $|\tilde F| = m$ and $E(\tilde F) \leq
  E(F)$. Then, again, repeatedly applying \cite[Lemma 3.6]{km:cpam12},
  we have $E \left( B_{\sqrt{m/\pi}}(0) \right) < E(\tilde F)$,
  indicating that $B_{\sqrt{m/\pi}}(0)$ is a minimizer of $E$ with
  mass $m$, a contradiction.
\end{proof}

We will need several additional properties related to disks as test
configurations. First we give an explicit formula for the potential
energy of a disk of radius $R$.

\begin{lem}[\cite{km:cpam12}, Corollary 3.5 ]
  \label{l:BE}
  We have
  \begin{align}
    \label{EBR}
    E(B_R(0)) = 2 \pi R + {2 \pi^2 \Gamma(2 - \alpha) \over \Gamma
      \left( 2 - {\alpha \over 2} \right) \Gamma \left( 3 - {\alpha
          \over 2} \right) } \, R^{4 - \alpha}.
  \end{align}
\end{lem}

\noindent Next, we introduce the potential associated with the unit
disk:
\begin{align}
  \label{vBR}
  v^B(|x|) := \int_{B_1(0)} {1 \over |x - y|^\alpha} \, dy.
\end{align}
An explicit computation shows that \cite[Lemma 3.8]{km:cpam12}
  \begin{align}
    \label{vbr}
    v^B(r) =
    \begin{cases}
      \left( {\pi \over r^\alpha} \right) \,
      _2F_1\left(\frac{\alpha}{2}, \frac{\alpha}{2};
        2;\frac{1}{r^2}\right), & r \geq 1, \vspace{1mm} \\
      \vspace{1mm} \left( {2 \pi \over 2-\alpha} \right) \,
      _2F_1\left(\frac{\alpha-2}{2}, \frac{\alpha}{2}; 1; r^2\right),
      & r < 1.
    \end{cases}
  \end{align}
  where $_2F_1(a, b; c; z)$ is the hypergeometric function
  \cite{abramowitz}.  We also have the following useful properties of
  $v^B$.

\begin{lem}
  \label{l:vB}
  We have
 \begin{align}
    \label{v0}
    v^B(0) = {2 \pi \over 2 - \alpha}.
  \end{align}
  If also $\alpha < 1$, we have $v^B \in C^1([0, \infty))$ and
  \begin{align}
    \label{v1}
    \left| {dv^B (1) \over dr} \right| = \max_{r \geq 0} \left| {dv^B
        (r) \over dr} \right| = \frac{\pi \alpha (2 - \alpha) \Gamma(1
      - \alpha)}{2 \Gamma^2 (2 - \frac{\alpha}{2})}.
  \end{align}
\end{lem}

\begin{proof}
  The formula in \eqref{v0} follows from \eqref{vbr} by noting that
  $_2F_1\left(\frac{\alpha-2}{2}, \frac{\alpha}{2}; 1; 0 \right) =
  1$. Smoothness of $v^B$ follows from \cite[Lemma 3.8]{km:cpam12}.
  Finally, to obtain \eqref{v1}, we differentiate the formula in
  \eqref{vbr} twice with respect to $r$. We get for $r > 1$:
  \begin{multline}
    {d^2 v^B \over dr^2} = \frac{1}{4} \pi \alpha r^{-\alpha-4} \\
    \times \left(4 (\alpha+1) r^2 \,
      _2F_1\left(\frac{\alpha}{2},\frac{\alpha+2}{2};2;\frac{1}{r^2}\right)+\alpha
      (\alpha+2) \,
      _2F_1\left(\frac{\alpha}{2}+1,\frac{\alpha}{2}+2;3;\frac{1}{r^2}\right)\right),
  \end{multline}
  and for $r < 1$:
  \begin{multline}
    {d^2 v^B \over dr^2} = -\frac{1}{4} \pi \alpha \left(\alpha
      (\alpha+2) r^2 \,
      _2F_1\left(\frac{\alpha}{2}+1,\frac{\alpha}{2}+2;3;r^2\right)+4
      \,
      _2F_1\left(\frac{\alpha}{2},\frac{\alpha+2}{2};2;r^2\right)\right).
  \end{multline}
  An inspection of these formulas shows that $v^B(r)$ is a concave
  function for $r < 1$ and a convex function for $r > 1$. Therefore,
  since $dv^B(0)/dr = 0$ and $dv^B(\infty)/dr = 0$, the derivative of
  $v^B(r)$ is negative for all $r > 0$ and reaches its absolute
  minimum at $r = 1$. The second equality in \eqref{v1} again follows
  from \cite[Lemma 3.8]{km:cpam12}.
\end{proof}

Finally, we introduce the isoperimetric deficit of a measurable set $F
\subset \mathbb R^2$:
\begin{align}
  \label{DF}
  D(F) := {P(F) \over \sqrt{4 \pi |F|}} - 1.
\end{align}
The following quantitative version of the isoperimetric inequality due
to Bonnesen will be useful \cite{bonnesen24} (see also
\cite{burago,osserman79,fuglede91}).

\begin{lem}[\cite{bonnesen24} ]
  \label{l:bonnesen}
  Let $F \subset \mathbb R^2$ be a convex open set which is
  bounded. Then there exists $x_0 \in \mathbb R^2$ and $r_1, r_2$
  satisfying $0 < r_1 \leq r_2$ such that $B_{r_1}(x_0) \subseteq F
  \subseteq B_{r_2}(x_0)$ and
  \begin{align}
    \label{eq:bonnesen}
    {(r_2 - r_1)^2 \over |F|} \leq (2 +D(F)) D(F) .
  \end{align}
\end{lem}

\noindent Note that for $D(F) \ll 1$ the constant in the right-hand
side of \eqref{eq:bonnesen} is optimal
\cite{bonnesen24,burago,osserman79,fuglede91}.

\section{An upper bound for the minimal energy}
\label{sec:Ebound}

In this section, we derive an ansatz-based upper bound for the minimal
energy scaling linearly with $m$ for large $m$, which will be useful
in a number of proofs.  Our ansatz consists of $n$ disks of equal
mass, spaced arbitrarily far apart.  We choose $n$ as a function of
$m$ to optimize our bound.

For now, it is convenient to work in terms of $R:=\sqrt{m/\pi}$.  We
shall switch back to $m$ in the final step.  Also, we shall focus on
bounding the energy per unit area; this will then yield a
corresponding energy bound in terms of $m$. We define the constant
\begin{align}
  V_0(\alpha) &:= V(B_1(0)) =
  \frac{2\pi^2\Gamma(2-\alpha)}{\Gamma(2-\frac{\alpha}{2})
    \Gamma(3-\frac{\alpha}{2})},
  \label{eq:C0}
\end{align}
which is just the potential energy of a unit ball.  Then
$E_1(R):=2{\pi}R+V_0(\alpha)R^{4-\alpha}$ denotes the energy of one
disk of radius $R$ by Lemma \ref{l:BE}.  We denote the infimum of the
energy obtained by splitting the mass $m={\pi}R^2$ into $n$ disks of
equal radius by $E_n(R):=n E_1(R / \sqrt{n})$. Here we noted that the
non-local interaction between these disks can be made arbitrarily
small by translating the disks sufficiently far apart.  Finally, we
define the corresponding energy per unit area
$\rho_n(R):=E_n(R)/({\pi}R^2)$.  For notational convenience we let
$\rho(R):=\rho_1(R)=E_1(R)/({\pi}R^2)$.  Note that
\begin{equation}\label{eq:rho_n}
  \rho_n(R)=nE_1(R/\sqrt{n})/({\pi}R^2)=\rho(R/\sqrt{n}).  
\end{equation}

To find our upper bound, we characterize the envelope of the graphs of
the sequence of functions which are appropriate dilations of
$\rho(R)$, the energy per unit area of a single disk of radius $R$.
Three of these functions are illustrated in Figure \ref{fig:rho}. From
this figure, one may suspect that the envelope can be determined
completely by locating the intersections of the adjacent graphs. This
is proved in the following sequence of lemmas, which are intended to
deal with the elementary, but rather tedious algebra involved.

\begin{figure}[t]
\begin{center}
  \includegraphics[width=4in]{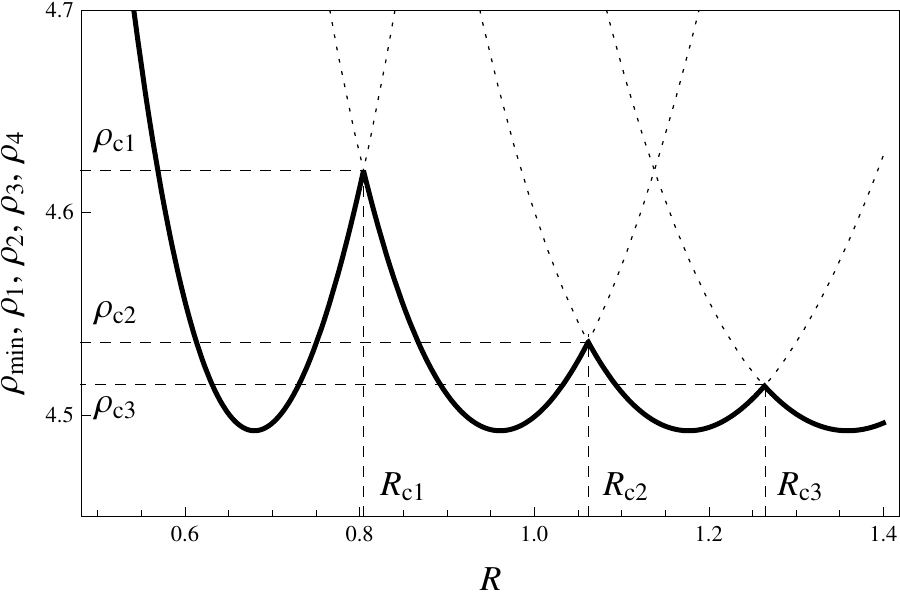}
\end{center}
\caption{The energy per unit area $\rho_n(R)$ for $n=1, 2, 3$ (dotted
  lines), along with the intersection points $(R_{cn}, \rho_{cn})$ and
  the minimum $\rhom(R)$ (solid line) for $\alpha = 0.1$.}
\label{fig:rho}
\end{figure}


\begin{lem}\label{lem:r_cn}
  For $n=1,2,\ldots$, the equation $\rho_n(R)=\rho_{n+1}(R)$ has a
  unique positive solution given by
\begin{align*}
  R_{cn} &:=\left( \frac{2\pi \left( \sqrt{n+1} - \sqrt{n}
      \right)}{V_0 \left( n^{{\alpha \over 2}-1} - (n+1)^{{\alpha
            \over 2}-1} \right)}\right)^{\frac{1}{3-\alpha}}.
\end{align*}
\end{lem}

\begin{proof}
  Since $R > 0$, we can solve the equivalent equation
  $E_n(R)=E_{n+1}(R)$. Solving
\begin{align*}
  n \left[ \frac{2\pi R}{\sqrt{n}} + V_0 \left( \frac{R}{\sqrt{n}}
    \right)^{4-\alpha} \right] = (n+1) \left[ \frac{2\pi
      R}{\sqrt{n+1}} + V_0 \left( \frac{R}{\sqrt{n+1}}
    \right)^{4-\alpha} \right]
\end{align*}
for $R$ gives the result.
\end{proof}

\begin{lem}\label{lem:intersection}

  Suppose the four functions $f_1, f_2, g_1, g_2 \in C^1([0,1])$
  satisfy
\begin{enumerate}
\item $f_i(0)=g_i(1)=0$ for $i=1, 2$;
\item $0<f_1 \sprime (x)<f_2 \sprime (x) \quad \forall x \in (0,1)$;
\item $0>g_1 \sprime (x)>g_2 \sprime (x) \quad \forall x \in (0,1)$.
\end{enumerate}
Then for $i=1,2$ there exist unique $x_i \in (0,1)$ such that
$f_i(x_i)=g_i(x_i)=:y_i$, with $y_1<y_2$.
\end{lem}

\begin{proof}
  Existence of unique $x_i$ for $i = 1,2$ follows from applying the
  intermediate value theorem to $h_i:=f_i-g_i$.  To
  prove 
  $y_1<y_2$, first consider $f_1(x)$ and $g_2(x)$.  These also have a
  unique intersection point; call it $(x_3,y_3)$. Define
  $h_3:=f_1-g_2$.  Then $h_3(0)=-g_2(0)=h_2(0)$, and $0<h_3 \sprime =
  f_1 \sprime - g_2 \sprime < f_2 \sprime - g_2 \sprime =h_2 \sprime$.
  So for $x \in (0,1)$, $h_2$ and $h_3$ are both increasing and
  $h_3<h_2$.  So $x_3$, the root of $h_3$, satisfies $x_3>x_2$.  Since
  $g_2$ is decreasing, we have $g_2(x_3)<g_2(x_2)$, which implies that
  $y_3<y_2$.  Similarly, by comparing the intersection of $f_1$ and
  $g_2$ to that of $f_1$ and $g_1$, we can show that $y_1<y_3$.
  Combining with the above, we get $y_1<y_2$.
\end{proof}

\begin{lem}\label{cor:slide}
  Let $f_1$, $g_1$ and $y_1$ be as in Lemma \ref{lem:intersection},
  and suppose $g_3(x)=g_1(x+a)$, where $0<a<1$. Then the unique
  solution $x_4 \in (0,1)$ of $f_1(x_4) = g_3(x_4) =: y_4$ satisfies
  $y_4<y_1$.
\end{lem}

\begin{proof}
  Analogous to that of Lemma \ref{lem:intersection}.
\end{proof}


\begin{lem}\label{lem:scale}
  Suppose $f_1 \in C^2(\RR^+)$ satisfies $f_1 \dprime >0$ and attains
  its minimum at $x_0$. Let $f_2(x):=f_1(x/a)$, where $a>1$.  Then
\begin{enumerate}[(i)]
\item $f_2$ satisfies $f_2  \dprime >0$ and attains its minimum at $ax_0$.
\item There is a unique $x_1 \in (x_0, ax_0)$ such that
  $f_1(x_1)=f_2(x_1)$.
\item For all $x<x_1$ we have $f_1(x) < f_2(x)$ and for all $x>x_1$ we
  have $f_1(x) > f_2(x)$.
\item Define a new function $f_3(x):=f_2\left(x+(a-1)x_0 \right)$
  which is $f_2$ shifted so its minimum coincides with that of $f_1$.
  Then $|f_3 \sprime| \leq |f_1 \sprime|$, with equality only when
  $x=x_0$.
\end{enumerate}
\end{lem}

\begin{proof}
  (i) is obvious.  To prove (ii), we use the fact that
  $f_1(x_0)=f_2(ax_0)$ and $f_1 \sprime >0, f_2 \sprime <0$ for all $x
  \in (x_0, ax_0)$.  By a slightly modified version of Lemma
  \ref{lem:intersection}, we know there is a unique intersection point
  in $(x_0,ax_0)$.  Also, we see that now (iii) certainly holds in
  $[x_0,ax_0]$.  The possibility of intersection \emph{outside} of
  this interval will be ruled out when we prove the rest of (iii).

  To prove $f_2(x) >f_1(x)$ for $x<x_0$, define the map $\sigma: (x,y)
  \mapsto (ax, y)$.  Then $\sigma$ maps the graph of $f_1$ on
  $(0,x_0/a)$ onto the graph of $f_2$ on $(0,x_0)$.  If $x \in
  (0,x_0/a)$, then $\sigma$ maps $(x,f_1(x))$ to a point above the
  graph of $f_1$, because $f_1$ decreases on $(0,x_0)$.  Similarly,
  since $f_1(x)$ increases for $x>x_0$, the image under $\sigma$ of
  its graph on this interval lies below its own graph.  But this image
  is the graph of $f_2(x)$ for $x>ax_0$.  So $f_2 (x) < f_1(x)$ for
  all $x>ax_0$.  This completes the proof of (iii).

  Finally, in part (iv), $f_1 \sprime(x_0)=f_3 \sprime(x_0)=0$.  Also,
  $f_3 \sprime(x) <0$ for $x<x_0$ and $f_3\sprime(x) >0$ for $x>x_0$.
  Then, by the chain rule
  \begin{align}
    \label{eq:fpx0a}
  f_3 \sprime (x) = f_2 \sprime \left( x+(a-1)x_0 \right) =
  \frac{1}{a} f_1 \sprime \left( \frac{x+(a-1)x_0}{a} \right)
  =\frac{1}{a} f_1 \sprime \left( x_0+\frac{x-x_0}{a} \right).
  \end{align}
  Suppose $x>x_0$.  Then
  \[
  x_0<x_0+\frac{x-x_0}{a}<x.
  \]
  Since $f_1 \sprime(x)$ is positive and increasing for $x>x_0$, by
  \eqref{eq:fpx0a} we get
  \[
  0<f_3 \sprime(x) < f_1 \sprime \left( x_0+\frac{x-x_0}{a} \right)<
  f_1 \sprime(x).
  \]
  Now suppose $x<x_0$.  Then
  \[
  x<x_0+\frac{x-x_0}{a}<x_0.
  \]
  Since $f_1 \sprime(x)$ is negative and increasing for $x<x_0$, we
  get analogously
  \[
  f_1 \sprime (x)< f_1 \sprime \left( x_0+\frac{x-x_0}{a} \right)< f_3
  \sprime (x) <0.
  \]
  Thus part (iv) is proved.
\end{proof}

\begin{lem}\label{lem:convex} 
  For $n=1,2,\ldots,$ and $\alpha \leq 1$, $\rho_n(R)$ has a positive
  second derivative for all $R > 0$ and attains the unique minimum at
  \[
  R_n:=\sqrt{n} \left( \frac{2\pi}{V_0(2-\alpha)}
  \right)^{\frac{1}{3-\alpha}}.
  \]
\end{lem}

\begin{proof}
  First, we prove the $n=1$ case by differentiating $\rho$ twice:
  \begin{align*}
    \rho(R) &= \frac{2}{R}+\frac{V_0}{\pi}R^{2-\alpha}\\
    \rho \sprime (R) &= -\frac{2}{R^2}+\frac{V_0}{\pi}(2-\alpha)R^{1-\alpha}\\
    \rho \dprime (R)&= \frac{4}{R^3}+\frac{V_0}{\pi
      R^\alpha}(2-\alpha)(1-\alpha).
  \end{align*}
  The second derivative is clearly positive for all $R>0$.  Hence
  $\rho(R)$ attains the unique minimum at $R = R_1$, where
  \begin{align*}
    R_1 = \left( \frac{2\pi}{V_0(2-\alpha)}
    \right)^{\frac{1}{3-\alpha}}.
  \end{align*}

  In the case of general $n > 1$, \eqref{eq:rho_n} and part (i) of
  Lemma \ref{lem:scale} yield the result.
\end{proof}

\begin{lem} \label{lem:rho_min} For $R>0$, let
  $\rhom(R):=\displaystyle \min_{n \in \N} \rho_n(R)$.  If we
  partition $\RR^+$ into disjoint intervals
  \begin{align*}
    I_1&:=(0, R_{c1}],
    \\
    I_n&:=(R_{c(n-1)}, R_{cn}], \qquad n=2,3, \ldots,
  \end{align*}
  then
  \[
  R \in I_n \implies \rhom(R)=\rho_n(R).
  \]
\end{lem}

\begin{proof}
  First, by part (ii) of Lemma \ref{lem:scale}, we have that $R_{cn}$
  lie between the successive minima of $\rho_n$, so they increase in
  $n$.  By part (iii) of Lemma \ref{lem:scale}, we have
  \begin{align}
    R<R_{cn} &\implies \rho_n(R) <\rho_{n+1}(R) \label{1stimpl}
    \\
    R>R_{cn} &\implies \rho_n(R) >\rho_{n+1}(R). \label{2ndimpl}
  \end{align}
  Suppose $R\leq R_{cn}$, $n \geq 1$.  Then $R\leq
  R_{cn}<R_{c(n+1)}<\cdots $, and repeatedly using \eqref{1stimpl}
  gives
  \[
  \rho_n(R)\leq \rho_{n+1}(R) \leq \cdots.
  \]
  The result of the Lemma for $R \in I_1$ then follows immediately.
  Otherwise, suppose $R>R_{c(n-1)}$, where $n>1$.  Then
  $R>R_{c(n-1)}>R_{c(n-2)}> \cdots$, so \eqref{2ndimpl} gives
  \[
  \rho_n(R) <\rho_{n-1}(R)<\cdots <\rho_1(R).
  \]
  Thus $R \in (R_{c(n-1)}, R_{cn}]$ implies that $\rho_n(R) \leq
  \rho_k(R)$ for every $k \in \mathbb N$, yielding the claim.
\end{proof}

\begin{lem}\label{lem:rho_cn}
  For $\alpha \leq 1$, the sequence $\rho_{cn}:=\rho_n(R_{cn})$ is
  decreasing in $n$.
\end{lem}

\begin{proof}
  Let $d_n:=R_{c(n+1)}-R_{cn}=R_{c1}(\sqrt{n+1}-\sqrt{n})$ be the
  distance between the successive minima described in Lemma
  \ref{lem:convex}.  Clearly $(d_n)$ is decreasing.  Now consider the
  graphs of functions $\rho_n(R)$ and $\rho_{n+1}(R)$, whose unique
  intersection point is $(R_{cn},\rho_{cn})$.  Shift these
  horizontally so that their minima are at $R=0$ and $R=d_n$, and call
  the new functions whose graphs these are as $f_2$ and $g_2$,
  respectively.  Note that the $\rho$-value of the intersection of
  these new functions is still $\rho_{cn}$.  Next, consider the graphs
  of $\rho_{n+1}(R)$ and $\rho_{n+2}(R)$, whose intersection point is
  $(R_{c(n+1)}, \rho_{c(n+1)})$.  As before, slide these so their
  respective minima are at $R=0$ and $R = d_n$, and call the new
  functions whose graphs these are as $f_1, g_1$.  By \eqref{eq:rho_n}
  and part (iv) of Lemma \ref{lem:scale}, $f_1, f_2, g_1, g_2$
  satisfy, up to translations and dilations, the hypotheses of Lemma
  \ref{lem:intersection}.  Thus, the intersection of the graphs of
  $f_1$ and $g_1$ lies below that of $f_2$ and $g_2$.

  Let $g_3$ be $g_1$ shifted to the left so that its minimum is now at
  $d_{n+1}$ rather than $d_n$.  By Lemma \ref{cor:slide}, the
  intersection point of the graphs of $f_1$ and $g_3$ is still below
  that of $f_2$ and $g_2$.  But the $\rho$-value of this intersection
  is $\rho_{c(n+1)}$.  Thus $\rho_{c(n+1)}<\rho_{cn}$ for every $n \in
  \mathbb N$.
\end{proof}

\begin{lem}\label{lem:rhobound}
  For $\alpha \leq 1$ and $R \geq R_{c1}$, we have $\rhom(R) \leq
  \rhoc$.
\end{lem}

\begin{proof}
  Suppose $R \geq R_{c1}$.  Since $\rhom(R)$ is convex between
  successive $R_{cn}$, it lies below the piecewise linear function
  connecting the $(R_{cn}, \rho_{cn})$ points.  By Lemma
  \ref{lem:rho_cn}, this function decreases.
\end{proof}

We are now able to prove the main result of this section. 

\begin{prop}[Upper Bound on Minimal Energy]
\label{prop:Escaling}
If $\alpha \leq 1$ and $\Omega$ is a minimizer of $E$ with mass $m
\geq m_{c1}$, where $m_{c1}$ is defined in \eqref{mc1}, then
\[
E(\Omega) \leq m \rhoc = {m \over m_{c1}} E \left( B_{\sqrt{m_{c1} /
      \pi}}(0) \right).
\]
\end{prop}

\begin{proof}
  In view of Lemma \ref{lem:rhobound}, testing $E$ with a union of $n
  \geq 1$ disks of mass $m/n$ sufficiently far apart and choosing $n$
  optimally, we get configurations whose energy can be made
  arbitrarily close to $m \rhoc$.  The desired inequality then
  follows.
\end{proof}

\section{Nonexistence of minimizers}
\label{sec:nonexistence}

The upper bound for the minimum of $E$ obtained in the preceding
section allows us to find a condition guaranteeing nonexistence of
minimizers as in \cite{km:cpam12}. Here, however, we will further
refine those estimates to ensure that the threshold value of $m$ for
non-existence approaches $m_{c1}(0) \approx 2.051$ as $\alpha \to 0$,
which is sharp.

We introduce an auxiliary function
\begin{align}
  \label{rho0}
  \rho_0(R) := {2 \over R} + {2^\alpha \pi^{1-\alpha} \over
    \rho_{c1}^\alpha} R^{2 - 2 \alpha},
\end{align}
where, as in Proposition \ref{prop:Escaling}, $\rho_{c1} = E_1(R_{c1})
/ ( \pi R_{c1}^2)$ and $R_{c1} = \sqrt{m_{c1} / \pi}$. Then we have
the following result concerning the roots of the equation $\rho_0(R) =
\rho_{c1}$ that exceed $R_{c1}$.

\begin{lem}
  \label{lem:rho0unique}
  For every $\alpha \leq \frac12$ there exists a unique value of $R_0
  \geq R_{c1}$ such that $\rho_0(R_0) = \rho_{c1}$. Moreover, if
  $R>R_{c1}$ then $\rho_0(R) > \rho_{c1}$ if and only if $R > R_0$.
\end{lem}

\begin{proof}
  First of all, observe that by \eqref{diamP} we have
  \begin{align*}
    \pi R_{c1}^2 \rho_{c1} = E(B_{R_{c1}}(0)) = P(B_{R_{c1}}(0)) +
    V(B_{R_{c1}}(0)) \\
    \geq P(B_{R_{c1}}(0)) +{\pi^2 R_{c1}^4 \over \left( \frac12
        P(B_{R_{c1}}(0)) \right)^\alpha}.
  \end{align*}
  Therefore, $P(B_{R_{c1}}(0)) < \pi R_{c1}^2 \rho_{c1}$ and
  \begin{align*}
    \pi R_{c1}^2 \rho_{c1} > 2 \pi R_{c1} + {2^\alpha \pi^{2-\alpha}
      \over \rho_{c1}^\alpha} R_{c1}^{4 - 2 \alpha} = \pi R_{c1}^2
    \rho_0(R_{c1}).
  \end{align*}
  Thus $\rho_0(R_{c1}) < \rho_{c1}$. On the other hand, since
  $\rho_0(R)$ is continuous and $\rho_0(R) \to +\infty$ as $R \to
  +\infty$, there exists $R_0 \geq R_{c1}$ such that $\rho_0(R_0) =
  \rho_{c1}$. Moreover, since
  \begin{align*}
    \rho_0''(R) = {4 \over R^3} + {2^{1+\alpha} \pi^{1-\alpha} \over
      \rho_{c1}^\alpha R^{2 \alpha}} (1 - \alpha) (1 - 2 \alpha) > 0
    \qquad \forall R > 0,
  \end{align*}
  the function $\rho_0(R)$ is strictly convex and, hence, the value of
  $R_0$ is unique. Finally, the last statement follows from the fact
  that $\rho_0(R) - \rho_{c1}$ changes sign from negative to positive
  as $R$ increases.
\end{proof}

We now state the nonexistence result. 

\begin{prop}[Nonexistence of Minimizers]
\label{thm:nonexistence}
Let $\alpha \leq \frac12$ and let $m_2 = \pi R_0^2$, where $R_0$ is as
in Lemma \ref{lem:rho0unique}. Then there is no minimizer of $E$ with
mass $m$ for any $m > m_2$.
\end{prop}

\begin{proof}
  Suppose, to the contrary, that a minimizer $\Omega$ exists and $m >
  m_2$. By Proposition \ref{prop:Escaling} and \eqref{diamP} we have
  \begin{align*}
    m \rho_{c1} \geq E(\Omega) = P(\Omega) + V(\Omega) \geq P(\Omega)
    + {m^2 \over \left( \frac12 P(\Omega) \right)^\alpha}.
  \end{align*}
  Therefore, we get $P(\Omega) \leq m \rho_{c1}$ and, hence, with the
  help of the isoperimetric inequality we obtain
  \begin{align*}
    \rho_{c1} \geq \sqrt{4 \pi \over m} + {2^\alpha m^{1-\alpha} \over
      \rho_{c1}^\alpha}  = \rho_0( \sqrt{m / \pi} ),
  \end{align*}
  which contradicts Lemma \ref{lem:rho0unique}.
\end{proof}

We note that the value of $m_2 = m_2(\alpha)$ in Proposition
\ref{thm:nonexistence} satisfies
\begin{align*}
  \lim_{\alpha \to 0} m_2(\alpha) = m_{c1}(0).
\end{align*}
This follows from the fact that the statement of Lemma
\ref{lem:rho0unique} remains valid up to $\alpha = 0$ and that
$\lim_{\alpha \to 0} \rho_0(R) = \rho(R)$ for every $R >
0$. Therefore, since $\rho(R)$ is strictly increasing when $R >
R_{c1}$, we have that $R_0 = R_{c1}$ in this case. The numerically
computed dependence of $m_2$ on $\alpha$ in Proposition
\ref{thm:nonexistence}, alongside with $m_{c1}(\alpha)$ from
\eqref{mc1}, is presented in Fig. \ref{f:m2}. This figure also shows
regions in the $(\alpha, m)$-plane in which minimizers of different
kinds are guaranteed to fail to exist.

\begin{figure}[t]
  \centering
  \includegraphics[width=4in]{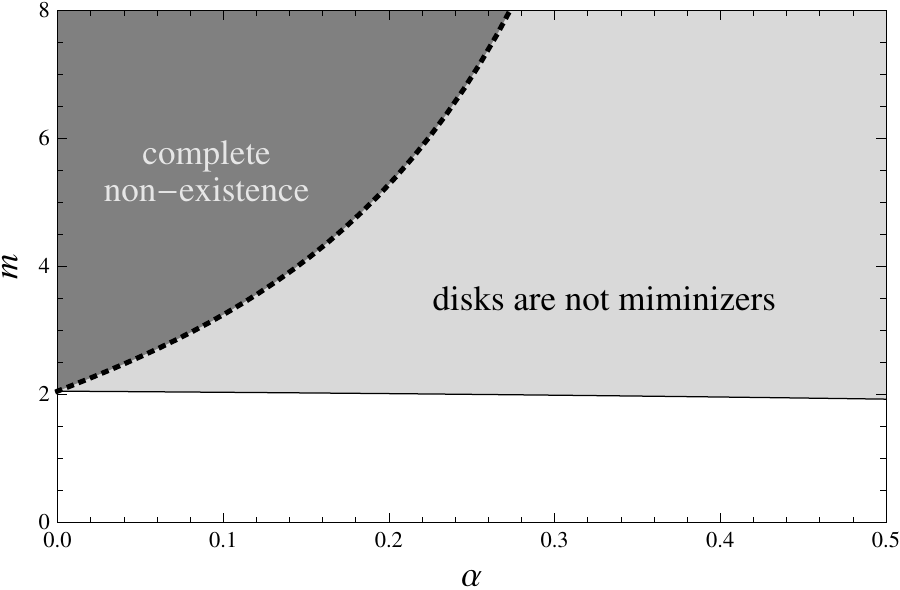}
  \caption{Regions of guaranteed non-existence of minimizers. The
    dotted line shows the plot of $m_2(\alpha)$ from Proposition
    \ref{thm:nonexistence} and the solid line shows the plot of
    $m_{c1}(\alpha)$ from \eqref{mc1}, both obtained numerically. Dark
    shaded area shows the region where minimizers of $E$ with mass $m$
    do not to exist. Light shaded area shows the region in which disks
    cannot be minimizers.}
  \label{f:m2}
\end{figure}

\section{Shape of minimizers}
\label{sec:shape}

We now investigate under which conditions the unique, up to
translations, minimizer of $E$ with mass $m$ is a disk. This is to be
expected in the regime of sufficiently small values of $m$ depending
on $\alpha$ \cite{km:cpam12}. Here we quantify this statement by
finding a mass $m_0 = m_0(\alpha)$ below which the minimizer, if it
exists, is a disk, and such that $m_0(\alpha)$ diverges as $\alpha \to
0$.

Let $\Omega$ be a minimizer of $E$ with mass $m$. In this section it
is convenient to introduce a rescaling $\Omega_\eps := \Omega
\sqrt{\pi / m}$ which ensures that $|\Omega_\eps| = |B_1(0)| =
\pi$. Here
\begin{align}
  \label{eq:epsilon}
  \eps := \left( {m \over \pi} \right)^{3 - \alpha \over 2},
\end{align}
is the new parameter, whose smallness implies smallness of $m$. In
terms of $\Omega_\eps$, we have $\sqrt{\pi / m} \, E(\Omega) =
E_\eps(\Omega_\eps)$, where \cite{km:cpam12}
\begin{align}
  \label{eq:Eeps}
  E_\eps(\Omega_\eps) := P(\Omega_\eps) + \eps V(\Omega_\eps),
\end{align}
and $\Omega_\eps$ is a minimizer of $E_\eps$ among all open bounded
sets with boundary of class $C^1$ and area equal to $\pi$. Throughout
the rest of this section, $\Omega_\eps$ always stands for a minimizer
of $E_\eps$.

We wish to estimate the range of values of $\eps > 0$ for which
$\Omega_\eps$ must be a unit disk. We proceed via a sequence of
lemmas.

\begin{lem}[Bound on isoperimetric deficit]
\label{lem:deficit}
If $D(\Omega_\eps)$ is the isoperimetric deficit of $\Omega_\eps$
defined in \eqref{DF}, then
\[
D(\Omega_\eps) < C_0, \qquad C_0(\alpha, \eps) := \frac{\eps}{2 \pi}
\left( V_0(\alpha) - {\pi^{2-\alpha} \over \left( 1 + {\eps \over 2
        \pi} V_0(\alpha) \right)^\alpha} \right) > 0.
\]
\end{lem}

\begin{proof}
  Testing the energy with a unit ball and using Lemma \ref{l:BE}, we
  obtain
  \begin{align*}
    P(\Omega_\eps) + \eps V(\Omega_\eps) = E(\Omega_\eps) \leq
    E(B_1(0)) = 2 \pi + \eps V_0.
  \end{align*}
  In particular, $P(\Omega_\eps) \leq 2 \pi + \eps V_0$ and,
  therefore, by \eqref{diamP} we have
  \begin{align*}
    V(\Omega_\eps) \geq {2^\alpha \pi^2 \over P^\alpha(\Omega_\eps)}
    \geq {\pi^{2-\alpha} \over \left( 1 + {\eps \over 2 \pi} V_0
      \right)^\alpha}.
  \end{align*}
  Combining the two inequalities above gives the result.
\end{proof}

\begin{rem}
  \label{rem:C0}
  Observe that $C_0(\alpha, \eps)$ is a monotonically increasing
  function of $\eps$ for $\alpha$ fixed, and that $C_0(\alpha, \eps)
  \to 0$ as $\alpha \to 0$ with $\eps$ fixed. In particular, in view
  of Fig. \ref{f:m2} minimizers of $E_\eps$, if they exist, are small
  perturbations of unit disks for $\alpha \ll 1$.
\end{rem}

\begin{lem}[Bounds on potential]
\label{lem:vbound}
Let 
\begin{align}
  \label{v}
  v(x) := \int_{\Omega_\eps} {1 \over |x - y|^\alpha} \, dy
\end{align}
be the potential associated with $\Omega_\eps$. Then we have 
\[
C_1 < v(x) < C_2 \qquad \forall x \in \overline\Omega_\eps,
\]
where
\[
C_1(\alpha, \eps) := \frac{\pi^{1-\alpha}}{(1 + C_0(\alpha, \eps)
  )^\alpha} \qquad \text{and} \qquad C_2(\alpha) := \frac{2 \pi}{2 -
  \alpha} .
\]
\end{lem}

\begin{proof}
  For any $x \in \overline \Omega_\eps$, let $v^B$ be as in
  \eqref{vBR}.  Then
\begin{align*}
  v^B(0) - v(x) & = \int_{B_1(x)} {1 \over \norm} \, dy -
  \int_{\Omega_\eps} {1 \over |x - y|^\alpha} \, dy \\
  &= \sint{B_1(x) \setminus \Omega_\eps}{\frac{1}{\norm}}{y} -
  \sint{\Omega_\eps \setminus B_1(x)}{\frac{1}{\norm}}{y}
  \\
  & > |B_1(x) \setminus \Omega_\eps| - |\Omega_\eps \setminus B_1(x)|
  = 0,
\end{align*}
since $|\Omega| = |B_1(x)|$.  Therefore, $v(x)$ is bounded from above
by $v^B(0)$, whose value is given by \eqref{v0}.

On the other hand, with the help of Lemma \ref{lem:deficit} and
\eqref{diamP} we obtain
\[
v(x) \geq {2^\alpha \pi \over P^\alpha(\Omega_\eps)} = {\pi^{1-\alpha}
  \over (1 + D(\Omega_\eps))^\alpha} > C_1,
\]
which yields the lower bound.
\end{proof}

\begin{lem}[Convexity]
\label{lem:convexity}
There exists a unique $\eps = \eps_0(\alpha) > 0$ which solves
\begin{align}
  \label{C9}
  \frac{1}{1 + C_0(\alpha, \eps)} + 2 \eps (C_1(\alpha, \eps) -
  C_2(\alpha)) = 0.
\end{align}
Furthermore, if $\eps < \eps_0$ then $\Omega_\eps$ is strictly convex.
\end{lem}

\begin{proof}
  First, in view of Remark \ref{rem:C0} observe that since $C_2 > C_1$
  and since $C_1(\alpha, \eps)$ is decreasing as a function of $\eps$,
  the left-hand side of \eqref{C9} is a monotonically decreasing
  continuous function of $\eps$. Therefore, existence of a unique
  solution of \eqref{C9} is guaranteed by the fact that its left-hand
  side approches unity as $\eps \to 0$, while tending to $-\infty$
  when $\eps \to +\infty$.

  By Proposition \ref{prp-basic} (after rescaling), the Euler-Lagrange
  equation for $\partial \Omega_\eps$ at $x \in \partial \Omega_\eps$
  is
  \begin{equation}
    \label{eq:EL}
    \kappa(x) + 2 \eps v(x) - \mu = 0,
  \end{equation}
  where $\kappa$ is curvature (positive if $\Omega_\eps$ is convex)
  and $\mu \in \mathbb R$ is the Lagrange multiplier.  Integrating
  \eqref{eq:EL} over the outer portion $\partial \Omega_\eps^0$ of the
  boundary $\partial \Omega_\eps$ with respect to arclength and using
  Lemmas \ref{lem:deficit} and \ref{lem:vbound} yields
  \[
  \mu = \frac{2 \pi}{|\partial \Omega_\eps^0|} + 2 \eps \bar{v} \geq
  \frac{2 \pi}{P(\Omega_\eps)} + 2 \eps \bar{v} > \frac{1}{1 + C_0} +
  2 \eps C_1,
  \]
  where $\bar{v}$ is the average of $v$ over $\partial \Omega_\eps^0$.
  Then by \eqref{eq:EL} and Lemma \ref{lem:vbound}, we have
  \[
  \kappa(x) = \mu - 2 \eps v(x) > \frac{1}{1 + C_0} + 2 \eps (C_1 -
  C_2),
  \]
  which is positive under the assumption of the Lemma.
\end{proof}

\begin{lem}[Confinement to an annulus]
\label{lem:annulus}
If $\Omega_\eps$ is convex, there exist $x_0 \in \mathbb R^2$ and
$\delta \geq 0$ such that
\begin{align*}
  B_{1-\delta}(x_0) \subseteq \Omega_\eps \subseteq B_{1+\delta}(x_0)
\end{align*}
and
\begin{align}
  \label{eq:delta}
  \delta \leq \sqrt{ \pi D(\Omega_\eps)(D(\Omega_\eps) + 2)},
\end{align}
with the convention that $B_{1-\delta}(x_0) = \varnothing$ if $\delta
\geq 1$.
\end{lem}

\begin{proof}
  By Lemma \ref{l:bonnesen}, there exist $x_0 \in \mathbb R^2$ and $0
  < r_1 \leq 1 \leq r_2$ such that $B_{r_1}(x_0) \subseteq \Omega_\eps
  \subseteq B_{r_2}(x_0)$ and $r_2 - r_1 \leq \sqrt{\pi D(\Omega_\eps)
    (D(\Omega_\eps) + 2)}$. Hence $1 - r_1 \leq \sqrt{\pi
    D(\Omega_\eps) (D(\Omega_\eps) + 2)}$ and $r_2 - 1 \leq \sqrt{\pi
    D(\Omega_\eps) (D(\Omega_\eps) + 2)}$, and the result follows.
\end{proof}

We are now ready to prove a quantitative criterion which guarantees
that the minimizer of $E_\eps$, if it exists and is convex, is a unit
disk for a given value of $\eps$. The proof follows the ideas in the
proof of \cite[Proposition 7.5]{km:cpam12} in a quantitative way.

\begin{prop}[Minimizers are disks]
\label{thm:disk}
There exists a unique $\eps = \eps_1(\alpha) > 0$ solving
\begin{align}
  \label{e1}
  \eps C_3 (\alpha, \eps) \Big[ \eps C_3 (\alpha, \eps) C_0 ( \alpha,
  \eps) \Big( C_0 (\alpha, \eps) + 2 \Big) + 2 \Big] = 1,
\end{align}
where
\begin{align}
  \label{C3}
  C_3 (\alpha, \eps) := {\pi^2 \alpha (2 - \alpha) \Gamma(1 - \alpha)
    \over 2 \Gamma^2(2 - {\alpha \over 2})} \left( 1 + \frac23
    \sqrt{\pi C_0(\alpha, \eps) (C_0(\alpha, \eps) + 2)} \right).
\end{align}
Furthermore, if $\eps < \eps_1$ and $\Omega_\eps$ is convex, then
$\Omega_\eps$ is a unit disk.
\end{prop}

\begin{proof}
  In view of Remark \ref{rem:C0}, the left-hand side of \eqref{C3}
  increases monotonically from zero to infinity as $\eps$ runs from
  zero to infinity. Hence there is a unique solution to \eqref{e1}.

  Now, testing $E_\eps$ with $B_1(x_0)$, where $x_0$ is as in Lemma
  \ref{lem:annulus}, gives
  \begin{align*}
    P(\Omega_\eps) + \eps V(\Omega_\eps) = E_\eps(\Omega_\eps) &\leq
    E_\eps(B_1(x_0)) = 2 \pi + \eps V_0,
  \end{align*}
  which is equivalent to
  \begin{align}
    D(\Omega_\eps) &\leq \frac{\eps}{2\pi} (V_0 - V(\Omega_\eps)) =:
    \frac{\eps}{2 \pi} \dv.
    \label{eq:D}
  \end{align}
  Combining \eqref{eq:delta} and \eqref{eq:D} then gives
  \begin{equation}
    \label{eq:deltanew}
    \delta^2 \leq  {\eps \over 2} \dv \left( {\eps \over 2 \pi} \dv + 2
    \right). 
  \end{equation}
  On the other hand, arguing as in \cite[Eqs. (7.14) and
  (7.15)]{km:cpam12} and applying Lemma \ref{l:vB}, we then find that
  \begin{align}
    \dv & \leq 2 \sint {B_1(x_0) \triangle \Omega_\eps}{|v^B(|x -
      x_0|) - v^B(1)|}{x} \notag \\
    & \leq 2 \left| {dv^B(1) \over dr} \right| \int_{B_{1+\delta}(x_0)
      \backslash B_1(x_0)} (|x - x_0| - 1) \, dx \notag \\
    & = 4 \pi \left| {dv^B(1) \over dr} \right| \int_0^\delta t (1 +
    t) dt \notag \\
    & = 2 \pi \left| {dv^B(1) \over dr} \right| \left( 1 + \frac23
      \delta \right) \delta^2,
\label{eq:dvboundnew}
\end{align}
where to arrive at the second line in \eqref{eq:dvboundnew} we
reflected all the points of the set $B_1 \backslash \Omega_\eps$ with
respect to $\partial B_1(x_0)$.  Furthermore, since in view of Lemmas
\ref{lem:deficit} and \ref{lem:annulus} we have
\begin{align}
  \label{del}
  \delta \leq \sqrt{\pi C_0 (C_0 + 2)},
\end{align}
from \eqref{eq:dvboundnew} and \eqref{v1} we get
\begin{align}
  \label{dv}
  \dv \leq 2 \delta^2 C_3.
\end{align}
Therefore, substituting the inequality in \eqref{dv} back to
\eqref{eq:deltanew} and then using \eqref{del} again yields that
either $\delta = 0$ or
\begin{align}
  \label{dvdv}
  \eps C_3 (\alpha, \eps) \Big[ \eps C_3 (\alpha, \eps) C_0 ( \alpha,
  \eps) \Big( C_0 (\alpha, \eps) + 2 \Big) + 2 \Big] \geq 1.
\end{align}
Since the latter is impossible by our assumption, the rest of the
proposition is proved.
\end{proof} 

\begin{figure}[t]
  \centering
  \includegraphics[width=4in]{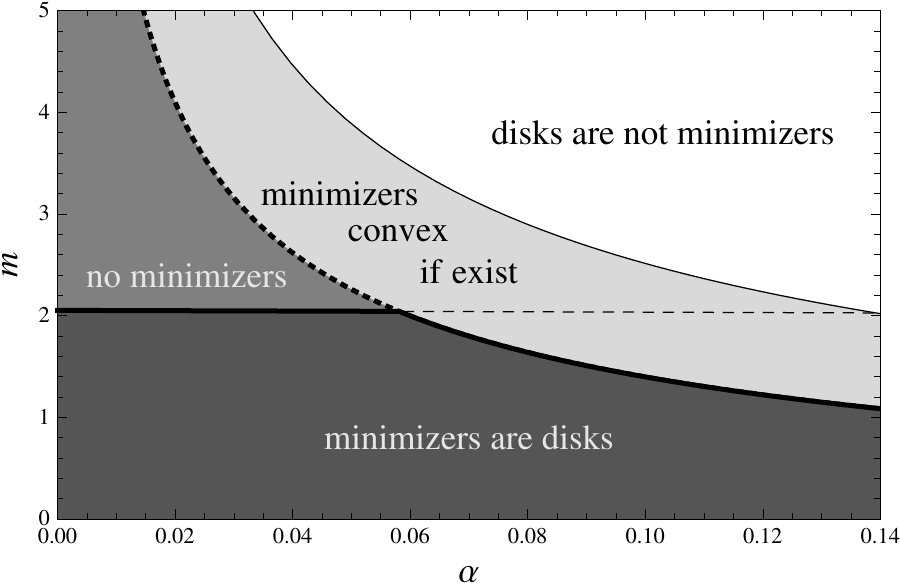}
  \caption{Regions of guaranteed convexity and existence of
    minimizers. The thin solid line shows the plot of
    $m(\eps_0(\alpha))$ from Lemma \ref{lem:convexity} and
    Eq. \eqref{eq:epsilon}, the dotted line shows the plot of
    $m(\eps_1(\alpha))$ from Proposition \ref{thm:disk} and
    Eq. \eqref{eq:epsilon}, the dashed line shows the plot of
    $m_{c1}(\alpha)$, and the thick solid line encloses the region in
    which minimizers exist and are disks. Light gray area shows the
    region where minimizers are convex, if they exist. Medium gray
    area shows the region where there are no minimizers. Dark gray
    area shows the region in which minimizers exist and are disks.}
  \label{f:m0}
\end{figure}

The dependences of $m(\eps_0)$ (thin solid line) and $m(\eps_1)$ on
$\alpha$ (dotted/solid line) computed numerically are presented in
Fig. \ref{f:m0}. These curves, together with the curve
$m_{c1}(\alpha)$ (dashed/solid line) separate the parameters into
several regions (see the caption for an explanation). Specifically,
the region below the thick solid line indicates the parameters for
which the minimizer of $E$ with mass $m$ exists and is a disk, while
the region above the solid line and below the dotted line is where no
minimizers exist. Our numerical results indicate that one can chose
$m_0(\alpha) = m(\eps_1(\alpha))$, using Proposition \ref{thm:disk}
and Eq. \eqref{eq:epsilon}.  Then for any mass $m \in
(0,m_0(\alpha))$, the minimizer, provided it exists, is a disk.

\section{Proof of the main Theorem}
\label{sec:results}

\begin{figure}[t]
\begin{center}
\includegraphics[width=4in]{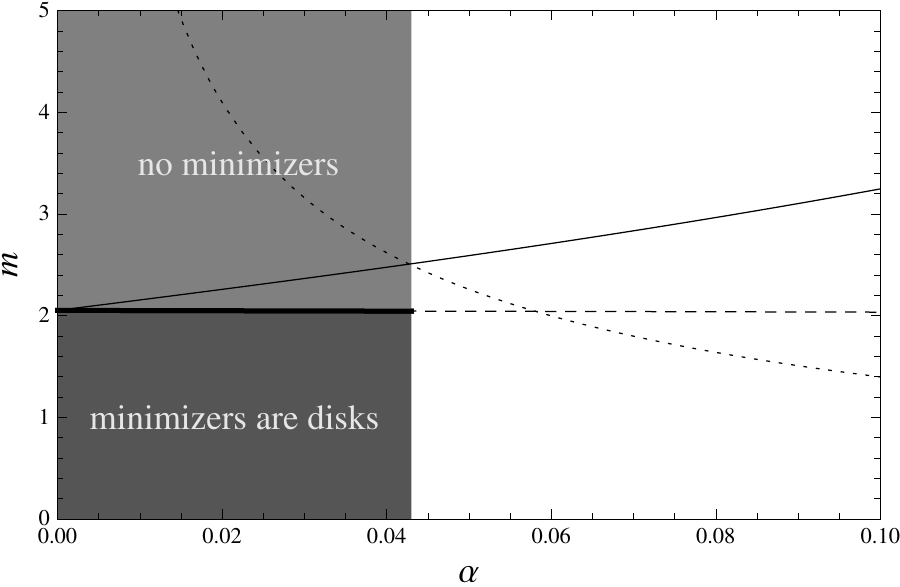}
\end{center}
\caption{Summary of the numerical results. Above the solid line no
  minimizers exist; below the dotted line minimizers, if they exist,
  are disks; above the dashed/thick solid line (lighter gray region)
  disks are not minimizers; below the solid line (dark gray region)
  minimizers are disks. }
\label{fig:alphac}
\end{figure}

Figure \ref{fig:alphac} summarizes our results obtained numerically
from evaluating the different criteria of existence and non-existence
obtained in the preceding sections. From this figure one can see that
the curve $m_2(\alpha)$, above which nonexistence of minimizers holds,
intersects the curve $m_0(\alpha)$, below which minimizers must be
disks, at $\alpha = \alpha_0 \approx 0.04273$. This indicates that the
statement of Theorem \ref{t:main} should hold below this value of
$\alpha$. In the rest of this section, we give an analytical proof of
this fact with a slightly reduced value of $\alpha_0$. The only
difficulty at this point is that the dependences of $m_0$, $m_2$ and
$m_{c1}$ on $\alpha$ are given by extremely complicated algebraic
formulas and, therefore, their qualitative behavior (e.g.,
monotonicity) is not easy to establish. Instead, we simply estimate
those functions explicitly for $\alpha \in (0,\alpha_0]$, using the
known behavior of the Gamma function and other functions appearing in
the estimates. Note that our analytical estimates below are rather ad
hoc and are not intended to be completely optimal. We believe that
$\alpha_0 = 0.0427$, which comes from our numerical results, should
give essentially the best constant with our approach. Proving this
fact would be an extremely tedious exercise in calculus, which we
decided not to pursue.

\begin{proof}[Proof of Theorem \ref{t:main}] 
  Since $\Gamma(z)$ is a monotonically increasing function of $z$ for
  $z \geq 1.966$ and a monotonically decreasing function of $z$ for $0
  < z \leq 1$, for all $0 < \alpha \leq 0.034$ we can bound its values
  that appear in our estimates as follows:
\begin{align*}
  \label{Gammas}
  0.986 & \leq \Gamma(2 - \alpha) \leq 1, \\
  0.992 & \leq \Gamma \left( 2 - {\alpha \over 2} \right) \leq 1, \\
  1.968 & \leq \Gamma \left( 3 - {\alpha \over 2} \right) \leq 2, \\
  1 & \leq \Gamma(1 - \alpha) \leq 1.021.
\end{align*}
Then from \eqref{mc1} we find that $2.007 \leq m_{c1}(\alpha) \leq
2.087$. 

Next, define (here and everywhere below the constants are as in the
previous sections, with the parametric dependences always indicated)
\begin{align*}
  F_1(\alpha,\eps):= \eps C_3(\alpha,\eps) \Big[ \eps C_3(\alpha,\eps)
  C_0(\alpha,\eps) \Big( C_0(\alpha,\eps) + 2 \Big) + 2 \Big] -1.
\end{align*}
Assume $\eps \leq 0.846$ and $0 < \alpha \leq 0.034$. Using the bounds
above, we then get
\begin{align*}
  C_0(\alpha, \eps) &\leq 0.121,
  \\
  C_3(\alpha, \eps) &\leq 0.557,
  \\
  F_1(\alpha, \eps) & < 0.
\end{align*}
By Proposition \ref{thm:disk}, it then follows that $\eps_1(\alpha) >
0.846,$ and, hence, $m(\eps_1) > 2.806,$ for $0 < \alpha \leq 0.034$.

Similarly, we can define 
\begin{align*}
F_2(\alpha, \eps):=\frac{1}{1 + C_0(\alpha, \eps)} - 2 \eps (C_2(\alpha) -
  C_1(\alpha, \eps) ).
 \end{align*}
 We find that if $\eps \leq 0.846$ and $ 0 < \alpha \leq 0.034$, then
\begin{align*}
  C_1(\alpha, \eps) &\geq 3.009,
  \\
  C_2(\alpha) &\leq 3.196,
  \\
  F_2(\alpha, \eps) &\geq 0.575 > 0.
\end{align*}
Therefore, by Lemma \ref{lem:convexity} we have $\eps_0(\alpha)
>0.846$ and, hence, $m(\eps_0) > 2.806$ for $0 < \alpha \leq 0.034$.

Finally, we wish to obtain an upper bound for $m_2(\alpha)$.  Define
\begin{align*}
  F_3(\alpha, R) := \rho_0(\alpha, R) - \rho_{c1}(\alpha),
\end{align*}
and recall that
\begin{align*}
  \rho_{c1}(\alpha) = {2 \over R_{c1}(\alpha)} + {V_0(\alpha) \over
    \pi} R_{c1}^{2 - \alpha}(\alpha).
\end{align*}
Then for $0 < \alpha \leq 0.034$ and $R=0.945$ we have
\begin{align*}
  0.799 &\leq R_{c1}(\alpha) \leq 0.815
  \\
  \rho_{c1}(\alpha) &\leq 4.656,
  \\
  \rho_0(\alpha, R) &\geq 4.677,
  \\
  F_3(\alpha, R) & \geq 0.021 > 0.
\end{align*}
Thus, by Proposition \ref{thm:nonexistence} and the arguments in the
proof of Lemma \ref{lem:rho0unique} we have $R_0(\alpha) < 0.945$ and,
hence, $m_2(\alpha) < 2.806$ for $0 < \alpha \leq 0.034$.

In conclusion, $m_2(\alpha) < \min \left( m(\eps_0), m(\eps_1)
\right)$ for every $\alpha \in (0, 0.034]$, which proves the result.
\end{proof}

\noindent \textbf{Acknowledgments.} The work of C. B. M. was supported,
in part, by NSF via grants DMS-0908279, DMS-1119724 and DMS-1313687.

\bibliography{../mura,../nonlin,../stat}

\begin{thebibliography}{10}

\bibitem{ambrosio}
L.~Ambrosio, N.~Fusco, and D.~Pallara.
\newblock {\em Functions of bounded variation and free discontinuity problems}.
\newblock Oxford Mathematical Monographs. The Clarendon Press, New York, 2000.

\bibitem{gamow30}
G.~Gamow.
\newblock Mass defect curve and nuclear constitution.
\newblock {\em Proc. Roy. Soc. London A}, 126:632--644, 1930.

\bibitem{bohr36}
N.~Bohr.
\newblock Neutron capture and nuclear constitution.
\newblock {\em Nature}, 137:344--348, 1936.

\bibitem{bohr39}
N.~Bohr and J.~A. Wheeler.
\newblock The mechanism of nuclear fission.
\newblock {\em Phys. Rev.}, 56:426--450, 1939.

\bibitem{km:cpam13}
H.~Kn\"upfer and C.~B. Muratov.
\newblock On an isoperimetric problem with a competing non-local term. {II.
  The} general case.
\newblock {\em Commun. Pure Appl. Math.}, 2013 (published online).

\bibitem{ohta86}
T.~Ohta and K.~Kawasaki.
\newblock Equilibrium morphologies of block copolymer melts.
\newblock {\em Macromolecules}, 19:2621--2632, 1986.

\bibitem{choksi11}
R.~Choksi and M.~A. Peletier.
\newblock Small volume fraction limit of the diblock copolymer problem: {II.
  Diffuse} interface functional.
\newblock {\em SIAM J. Math. Anal.}, 43:739--763, 2011.

\bibitem{choksi12}
R.~Choksi.
\newblock On global minimizers for a variational problem with long-range
  interactions.
\newblock {\em Quart. Appl. Math.}, LXX:517--537, 2012.

\bibitem{km:cpam12}
H.~Kn\"upfer and C.~B. Muratov.
\newblock On an isoperimetric problem with a competing non-local term. {I. The}
  planar case.
\newblock {\em Comm. Pure Appl. Math.}, 66:1129--1162, 2013.

\bibitem{lu14}
J.~Lu and F.~Otto.
\newblock Nonexistence of minimizer for {Thomas-Fermi-Dirac-von Weizs\"acker}
  model.
\newblock {\em Comm. Pure Appl. Math.}, 2013 (publlished online).

\bibitem{cicalese13}
M.~Cicalese and E.~Spadaro.
\newblock Droplet minimizers of an isoperimetric problem with long-range
  interactions.
\newblock {\em Comm. Pure Appl. Math.}, 66:1298--1333, 2013.

\bibitem{julin12}
V.~Julin.
\newblock Isoperimetric problem with a {Coulombic} repulsive term.
\newblock {\em Indiana Univ. Math. J.}, 2014 (to appear).
\newblock Preprint: arXiv:1207.0715.

\bibitem{bonacini13}
M.~Bonacini and R.~Cristoferi.
\newblock Local and global minimality results for a nonlocal isoperimetric
  problem on {$\mathbb R^N$}.
\newblock Preprint: arXiv:1307.5269, 2013.

\bibitem{acerbi13}
E.~Acerbi, N.~Fusco, and M.~Morini.
\newblock Minimality via second variation for a nonlocal isoperimetric problem.
\newblock {\em Commun. Math. Phys.}, 322:515--557, 2013.

\bibitem{sternberg11}
P.~Sternberg and I.~Topaloglu.
\newblock A note on the global minimizers of the nonlocal isoperimetric problem
  in two dimensions.
\newblock {\em Interfaces Free Bound.}, 13:155--169, 2010.

\bibitem{topaloglu13}
I.~Topaloglu.
\newblock On a nonlocal isoperimetric problem on the two-sphere.
\newblock {\em Comm. Pure Appl. Anal.}, 12:597--620, 2013.

\bibitem{figalli14}
A.~Figalli, N.~Fusco, F.~Maggi, V.~Millot, and M.~Morini.
\newblock Isoperimetry and stability properties of balls with respect to
  nonlocal energies.
\newblock Preprint: arXiv:1403.0516, 2014.

\bibitem{abramowitz}
M.~Abramowitz and I.~Stegun, editors.
\newblock {\em Handbook of mathematical functions}.
\newblock National Bureau of Standards, 1964.

\bibitem{ambrosio98}
L.~Ambrosio and E.~Paolini.
\newblock Partial regularity for quasi minimizers of perimeter.
\newblock {\em Ricerche Mat.}, 48(supplemento):167--186, 1998.

\bibitem{tamanini84}
I.~Tamanini.
\newblock Regularity results for almost minimal oriented hypersurfaces in
  {$\mathbb R^N$}.
\newblock {\em Quaderni Dipartimento Mat. Univ. Lecce}, 1:1--92, 1984.

\bibitem{bonnesen24}
T.~Bonnesen.
\newblock \"{U}ber das isoperimetrische {D}efizit ebener {F}iguren.
\newblock {\em Math. Ann.}, 91:252--268, 1924.

\bibitem{burago}
Yu.~D. Burago and V.~A. Zalgaller.
\newblock {\em Geometric inequalities}.
\newblock Springer-Verlag, New York, 1988.

\bibitem{osserman79}
R.~Osserman.
\newblock Bonnesen-style isoperimetric inequalities.
\newblock {\em Amer. Math. Monthly}, 86:1--29, 1979.

\bibitem{fuglede91}
B.~Fuglede.
\newblock Bonnesen's inequality for the isoperimetric deficiency of closed
  curves in the plane.
\newblock {\em Geom. Dedicata}, 38:283--300, 1991.

\end{thebibliography}

\bibliographystyle{unsrt}

\end{document}